\documentclass[11pt]{article}
\usepackage[utf8]{inputenc}
\usepackage{color}
\usepackage{amsthm}
\usepackage{mathrsfs}
\usepackage{stmaryrd}
\usepackage{microtype}
\usepackage[T1]{fontenc}
\usepackage{authblk}
\usepackage{stackengine}
\usepackage{amscd,verbatim,fancyhdr}
\usepackage{graphicx}
\usepackage{turnstile}
\usepackage[reqno, namelimits, sumlimits]{amsmath}
\newcommand{\ee}{\epsilon}

\newcommand{\R}{\mathbb{R}}

\DeclareMathOperator*{\dvg}{\mbox{div}}

\usepackage{amssymb,amsfonts}
\usepackage{bbold}
\usepackage{multirow}

\usepackage{blindtext}
\usepackage{hyperref}
\hypersetup{
    colorlinks=true,
    linkcolor=green,
}
\newtheorem{theorem}{Theorem}[section]
\newtheorem{proposition}{Proposition}[section]

\newtheorem{lemma}{Lemma}

\theoremstyle{remark}
\newtheorem{remark}{Remark}
\theoremstyle{definition}
\newtheorem{definition}{Definition}[section]

\begin{document}

\title{On a Toy-Model related to the Navier-Stokes Equations}

\author{F.~Hounkpe\footnote{Email address: \texttt{hounkpe@maths.ox.ac.uk}}
}
\affil{Mathematical Institute, University of Oxford, Oxford, UK}

\maketitle

\begin{abstract}
In this paper we consider a parabolic toy-model for the incompressible Navier-Stokes system. This model, as we shall see below, shares a lot of similar features with the incompressible model; among which the energy inequality, the scaling symmetry, and it is also supercritical in $3$D. Our goal is to establish some regularity results for this toy-model in order to get, if possible, better insight to the standard Navier-Stokes system. We also prove here, in a direct manner, a Caffarelli-Kohn-Nirenberg type result for our model. Finally, taking advantage of the absence of the divergence-free constraint, we are able to study this model in the radially symmetric setting for which we are able to establish full regularity.  
\end{abstract}
\setcounter{equation}{0}
\section{Introduction}
One of the fundamental questions in the mathematical hydrodynamics is the global well-posedness of the $3$D incompressible Navier-Stokes system, i.e. the global existence of a unique solution, for a given smooth divergence-free initial data. Despite the effort of many mathematicians, this question still remains unanswered. Nevertheless, a lot of progresses have been made, which allow us to better understand the regular or singular behaviour of this system.\\
We can enumerate many reasons why this problem is so difficult in $3$D. But, the most notable one is its supercriticality and we have a poor understanding of supercritical equations. And by supercritical, we mean that the globally controlled quantities available for the system (in the case of the incompressible Navier-Stokes system, those are the kinetic energy and the dissipation) are very weak or do not control at all the solution when we move down to smaller scales (or in other word when we zoom-in on the solution). Another major difficulty that should also be mentioned is the non-locality (characterised by the presence of the pressure) introduced by the incompressibility condition. In order to tackle the latter difficulty, one idea would be to find an approximate system to the Navier-Stokes equations which in a sense is completely local, study the regularity of solution(s) of this approximating system and hope to conserve this regularity (if there is) in the limit. This has partially been done by Rusin in \cite{Rus12}. He considered the following approximating model:
\[
\partial_t u - \Delta u - \frac{1}{\epsilon}\nabla \dvg u + u\cdot \nabla u + \frac{1}{2}u \dvg u = 0\quad (\mbox{$\epsilon > 0$}).
\]
He successfully proved convergence of a sequence of solution of the approximating system to a solution of the incompressible Navier-Stokes system but was unable to establish a complete regularity result for his approximating model. Our goal is to take a small step toward that realisation. We consider and hope to get a complete understanding of that system, where the bulk viscosity term "$\epsilon^{-1}\nabla \dvg u$" is removed:
\begin{equation}\label{Toy-NS1}
\partial_t u - \Delta u + u\cdot \nabla u + \frac{1}{2}u \dvg u = 0,
\end{equation}
with $u=(u_i)_{i=1}^n$ a vector field in $\R^n$.\\
It is quite straightforward to check that a smooth and rapidly decaying solution $u$ to system \eqref{Toy-NS1} has the following energy identity
\begin{equation}\label{I-E2}
    \frac{1}{2}\int_{\R^3}|u(x,t)|^2 dx + \int_0^t \int_{\R^3} |\nabla u(x,s)|^2dxds = \frac{1}{2}\int_{\R^3}|u_0(x)|^2 dx,
\end{equation}
which holds also true for the incompressible Navier-Stokes equations. Moreover, this toy-model has the same scaling symmetry as the incompressible Navier-Stokes system i.e. for any solution $u(x,t)$ to system \eqref{Toy-NS1}, we have that $u^{\lambda}(x,t) := \lambda u(\lambda x, \lambda^2 t)$ (for $\lambda >0$) is also solution to \eqref{Toy-NS1}; and this scaling symmetry is the only one we have for this system. With that in mind and going back to the energy identity \eqref{I-E2}, we get that our toy-model is (like the incompressible Navier-Stokes system) supercritical in dimension $3$. Indeed, we have:
\[ \int_{\R^3}|u^{\lambda}(x,T)|^2 dx, \int_0^T \int_{\R^3} |\nabla u^{\lambda}(x,s)|^2 dxds = O(1/{\lambda})\mbox{ as }\lambda \to 0. \]
We are not the first ones to study the regularity question for a toy-model to the Navier-Stokes equations. In fact, this question has been extensively examined (see for instance \cite{Tao16,Mont01,Gal09,Li08}). Among those works done on this subject, it's worth mentioning the one by Tao in \cite{Tao16} who successfully proved existence of a finite time blowup for a model that satisfies, unlike in the other papers, the energy identity \eqref{I-E2}. It is also worth mentioning that his system doesn't have the special structure of the nonlinearity in the Navier-Stokes system and therefore does not recover some of its fine properties, such as e.g. Caffarelli-Kohn-Nirenberg results (see \cite{Caff82}) and the backward uniqueness (see \cite{Esc03}) which, we are able to establish for our models. Therefore when it comes to gaining a better understanding of the regular and singular behaviour of the incompressible Navier-Stokes equations, our toy-model appears to be a suitable next step following the work of Tao.\\
As announced at the beginning of this paper, the goal here is to present some partial regularity results that might not occur in the incompressible model and discuss the radial symmetry case where one can completely answer the question of regularity for our model. Unfortunately, we are not yet able to fully answer that question in the general case. It's also worth mentioning that our partial regularity results hold true for a wider class of model; to be more precise, equations of the form
\[ \partial_t u - \Delta u + S(u,\nabla u) = 0,\]
where $S: \R^n\times \R^{n\times n} \to \R^n$ is bilinear and identity \eqref{I-E2} holds also true. For the radial symmetry case, an explicit knowledge of the structure of the non linearity is necessary (and this structure should be also adequate) in order to conclude, but the methodology to do so is similar to the one we present here.
\section{Preliminaries}
Before continuing our development, let us explain our notations
\[ z=(x,t), \quad z_0=(x_0,t_0),\quad B(x_0,R) = \{ |x-x_0|<R \};\]
\[ B_+(x_0,R) = \{ x\in B(x_0,R): x_{03}>0\};\]
\[ Q(z_0,R) = B(x_0,R)\times (t_0-R^2,t_0),\quad Q_+(z_0,R) = B_+(x_0,R)\times (t_0-R^2,t_0); \]
\[ B(r) = B(0,r),\quad Q(r) = Q(0,r),\quad B=B(1),\quad Q = Q(1); \]
\[ B_+(r) = B_+(0,r),\quad Q_+(r) = Q_+(0,r),\quad B_+=B_+(1),\quad Q_+ = Q_+(1) \]
For $\Omega$ an open subset of $\R^n$ and $-\infty\leq T_1<T_2 \leq +\infty$. Set $Q_{T_1,T_2}:= \Omega \times (T_1,T_2)$. We will be using
$L_{m,n}(Q_{T_1,T_2}):= L_n(T_1,T_2;L_m(\Omega))$, the Lebesgue space with the norm
\[
\|v\|_{m,n,Q_{T_1,T_2}} = \begin{cases}
\left( \int_{T_1}^{T_2}\|v(\cdot,t)\|^n_{L_m(\Omega)}dt\right)^{1/n},\quad & 1\leq n< \infty\\
\mbox{ess}\displaystyle\sup_{(T_1,T_2)}\|v(\cdot,t)\|_{L_m(\Omega)},\quad & n= \infty,
\end{cases}
\]
\[ L_m(Q_{T_1,T_2})=L_{m,m}(Q_{T_1,T_2}), \quad \|v\|_{m,m,Q_{T_1,T_2}} = \|v\|_{m,Q_{T_1,T_2}}; \]
$W^{1,0}_{m,n}(Q_{T_1,T_2})$, $W^{2,1}_{m,n}(Q_{T_1,T_2})$ are the Sobolev spaces with mixed norm,
\[ W^{1,0}_{m,n}(Q_{T_1,T_2}) = \left\{ v,\nabla v \in L_{m,n}(Q_{T_1,T_2}) \right\}, \]
\[ W^{2,1}_{m,n}(Q_{T_1,T_2}) = \left\{ v,\nabla v, \nabla^2 v, \partial_t v \in L_{m,n}(Q_{T_1,T_2}) \right\}, \]
\[ W^{1,0}_m(Q_{T_1,T_2}) = W^{1,0}_{m,m}(Q_{T_1,T_2}), \quad W^{2,1}_m(Q_{T_1,T_2}) = W^{2,1}_{m,m}(Q_{T_1,T_2}).\]
For various mean values of functions, we write
\begin{gather*}
    [f]_{x_0,R} := \frac{1}{|B(R)|}\int_{B(x_0,R)} f(x) dx,\quad (f)_{z_0,R} = \frac{1}{|Q(R)|}\int_{Q(z_0,R)}g(z)dz\\
    [f]_{,R} := [f]_{0,R},\quad (g)_{,R} = (g)_{0,R}.
\end{gather*}
Here $|\omega|$ and $|\Omega|$ stands for the $3$ and $4$-dimensional Lebesgue measure of the domains $\omega$ and $\Omega$ respectively. Lastly, we denote for simplicity, $f_{,i} := \partial_i f$ and summation over repeated indices running from $1$ to $3$ is adopted.\\
We give in what follows, the right notion of solutions needed for our work.
\begin{definition}\label{def2.1}
Let $\omega$ be a domain of $\R^3$. We say that $u$ is a \textit{suitable weak solution} to \eqref{Toy-NS1} in $\omega \times (T_1,T_2)$ if $u$ obeys the following conditions:
\begin{enumerate}
    \item $u\in L_{2,\infty}(\omega\times (T_1,T_2))\cap L_2(T_1,T_2;W_2^1(\omega))$
          and satisfies system \eqref{Toy-NS1} in the sense of distributions;
    \item The following local energy inequality holds
    \begin{multline*}
     \int_{\omega} \phi |u(x,t)|^2 dx + 2\int_{T_1}^t\int_{\omega} \phi|\nabla u|^2 dx ds \leq   \int_{T_1}^t\int_{\omega} |u|^2 (\partial_t \phi + \Delta \phi ) dx ds \\+ \int_{T_1}^t\int_{\omega} u\cdot  \nabla \phi |u|^2 dx ds, 
    \end{multline*}
    holds for a.e $t\in (T_1,T_2)$ and all non-negative functions $\phi \in C^{\infty}_0(\omega\times (T_1,\infty))$.
\end{enumerate}
\end{definition}
Like in the case of the standard Navier-Stokes system, we do not know whether or not every weak solution of our model (i.e. a solution that belongs to the energy class prescribed in the first point of the previous definition) verifies the above local energy inequality, and this is one of the motivation for considering such a subclass of weak solutions. Further reasons will be discussed below.\\
For the partial boundary regularity, the right notion of solutions in given by the following. For the sake of simplicity, we give the definition for the case of the canonical cylinder, but this can be trivially extended to more general cylinders. 
\begin{definition}\label{def2.2}
We say that the function $u$ is a \textit{boundary suitable weak solution} to \eqref{Toy-NS1} in $Q_+$ if:
\begin{enumerate}
    \item $u\in L_{2,\infty}(Q_+)\cap W^{1,0}_2(Q_+)$, and satisfies \eqref{Toy-NS1} in the sense of distributions in $Q_+$ with the following no-slip boundary condition 
    \[
     u|_{x_3 = 0} = 0;
    \]
    in the sense of the trace;
    \item The following local energy inequality holds true
    \begin{multline}
    \int_{B_+} \phi |u(x,t)|^2 dx + 2\int_{-1}^t\int_{B_+} \phi|\nabla u|^2 dx ds \leq \int_{-1}^t\int_{B_+} |u|^2 (\partial_t \phi + \Delta \phi ) dx ds \\+ \int_{-1}^t\int_{B_+} u\cdot \nabla \phi |u|^2 dx ds, 
    \end{multline}
    for a.e $t \in (-1,0)$ with $\phi \in C^{\infty}_0(B\times (-1,1))$.
\end{enumerate}

\end{definition}
As additional motivation to consider those suitable weak solutions, we mention the following. Firstly, among the energy solutions of our model (those are for our toy-model the equivalents of the Navier-Stokes system's weak Leray-Hopf solutions) there is at least one suitable weak solution (the construction of such solution can be done in the same as in the case of Navier-Stokes system). Secondly, similarly to the Navier-Stokes system, there are strong reasons to believe that smooth energy solutions to our model are to be sought among the suitable weak solutions. And thirdly, we are also able to connect the question of smoothness and uniqueness for our toy-model. These last two points will be discussed in more detail elsewhere.\\
We are now ready to state the main results of this paper.
\section{Main Results}
The first three main results of this paper are partial regularity results. By this we mean, assuming an extra control on a certain norm of the solution, we aim to derive further regularity properties. The class of quantities at the heart of this analysis are called \textit{scale invariant quantities}, which are quantities $F(u,r)$ (with $u$ a solution)
such that:
\[ F(u^{\lambda},1) = F(u,\lambda), \]
for all $\lambda>0$ and $u^{\lambda} = \lambda u(\lambda x, \lambda^2 t)$. Our first result in that direction states as follows.
\begin{theorem}[Higher space-time integrability for the gradient]\label{Thm3.1}
Let $u$ be a suitable weak solution to \eqref{Toy-NS1} in $Q$ such that 
\[ 
\mbox{ess}\sup_{-1< t < 0} \|u(\cdot,t)\|_{BMO^{-1}(B)} < \infty. 
\]
Then we have that 
\[ \int_{Q(1/2)}|\nabla u|^{2+\delta} < \infty,\]
with $\delta>0$. Here, $f\in BMO^{-1}(B; \R^3)$ shall be understood as there exists $F \in BMO(B; \R^{3\times 3})$ such that $f = \dvg F$. (We do not need anything more than this property in our proof but the interested reader may find more details about the space $BMO^{-1}$ in \cite{Koch01}) 
\end{theorem}
This assumption $u\in L_{\infty}(-1,0;BMO^{-1}(B))$ arise naturally when one studies type I blow-ups for the Navier-Stokes equations. These connections will be presented elsewhere. We discuss the consequence of this result in the last section of this paper. It is also worth mentioning that there is no such higher integrability result for the $3$D incompressible Navier-Stokes system at this time, and this is mainly due to the presence of the pressure (see e.g. \cite{Ser09} where this problem was considered). 

Our next main result is a Caffarelli-Kohn-Nirenberg type theorem for our model. The proof can done following ideas developed for the case of the Navier-Stokes system (see e.g. \cite{Ser14,Lin98,Caff82}) but we present here a more direct approach following ideas from \cite{Ser03}. The advantage of this method is that it gives us an estimation of the smallness parameter (see proof of Proposition \ref{Prop5.1} and Remark \ref{V-R1} below).
\begin{theorem}[A Caffarelli-Kohn-Nirenberg type result]\label{Thm3.2}
There exists a positive constant $\epsilon$ such if for any suitable weak solution $u$ to \eqref{Toy-NS1} in $Q$, we have 
\begin{equation}
    \sup_{0<r<1} \frac{1}{r}\int_Q |\nabla u|^2 < \epsilon,
\end{equation}
then the map $z\mapsto u(z)$ is h\"older continuous in $\overline{Q(\frac{1}{2})}$. Moreover, there exist absolute positive constants $c_{k,l}$ ($k,l = 0,1,2,\ldots$) such that
\[ \max_{\overline{Q(1/2)}}|\partial^l_t\nabla^k u(z)| \leq c_{k,l} \]
\end{theorem}
Before stating the next main result let us emphasise that, unlike the case of the standard Navier-Stokes system, this Caffarelli-Kohn-Nirenberg type result gives us for our model smoothness in time; and this mainly due to the absence of pressure.\\
Our next main result is concerned with partial boundary regularity and states as follows.
\begin{theorem}\label{Thm3.3}
Let $u$ be a boundary suitable weak solution in $Q_+$ (See Definition \ref{def2.2}) such that $u \in L_{\infty}(Q_+)$. Then, we have 
\[ u \in C^{\infty}(\overline{Q_+(a)}), \]
with $0<a<1$.
\end{theorem}
It is worth mentioning that the previous high regularity (H\"older continuity of higher order derivatives) does not necessarily occur in the case of the 3D incompressible Navier-Stokes system for which a counter-example was constructed (see \cite{Ser10}).
Our last main result is about the case where our suitable weak solution is radially symmetric. Let us point our that the divergence-free condition prevents such situation to occur in the case of the incompressible Navier-Stokes system. We are able to prove regularity of such solution in that case for our model. This gives us, to the contrary of the previous two theorems, a geometric condition for obtaining regularity. It is also worth mentioning that a similar result was proved by \v{S}ver\'{a}k and co-author in \cite{Ple03} for the whole space and relied on the decay at infinity of the solution. Our proof, on the contrary is for the local setting and the ideas therein can easily be applied to more diverse cases e.g. axially symmetric case with radial dependence only (this will be presented elsewhere).
\begin{theorem}\label{Thm3.4}
Let $u$ be a suitable weak solution to \eqref{Toy-NS1}, which is moreover radially symmetric i.e. \[ u(x,t) = -v(|x|,t)x. \]
Then, we have that 
\[ u\in C^{\infty}(Q). \]
\end{theorem}
In the following sections, we give the proof of the theorems stated above.
\section{Proof of Theorem~\ref{Thm3.1}}
\paragraph{Step I.} We begin by establishing the following Caccioppoli's type inequality. To formulate it, we need to introduce additional notations. Fix non negative cut-off functions $\varphi \in C^{\infty}_0(B(2))$ such that  $\varphi \equiv 1$ in $B$ and $\chi$ with the following properties:
\[
\chi(t) = \begin{cases}
0 \mbox{ for } t\leq -4,\\
(t + 4)/3 \mbox{ for } - 4 < t\leq-1,\\
1 \mbox{ for }t > -1.
\end{cases}
\]
Now, for an arbitrary point $z_0 = (x_0,t_0)$ such that $Q(z_0,2R) \subset Q$, we set 
\[ 
\chi_{t_0,2R}(t) = \chi((t-t_0)/R^2), \quad \varphi_{x_0,2R}(x) = \varphi((x-x_0)/R),
\]
and, as in \cite{Giaq82}, we introduce the special mean value
\[
u_{x_0,2R}(t) = \int_{B(x_0,2R)}u(x,t) \varphi^2_{x_0,2R}(x) dx \left( \int_{B(x_0,2R)} \varphi^2_{x_0,2R}(x)dx \right)^{-1}.
\]
Set  $\hat{u}_{x_0,2R}(x,t):= u(x,t) - u_{x_0,2R}(t)$ and introduce the matrix valued function $F = (F_{ij}) \in L_{\infty}(-1,0;BMO(B;\R^{3\times 3}))$ which is such that $u = \dvg F$.\\
Our Caccioppoli's inequality reads as follows
\begin{multline}\label{IV-E5}
    \int_{B(x_0,2R)} |\hat{u}_{x_0,2R}(x,t_0')|^2 \varphi^2_{x_0,2R}(x)\chi_{t_0,2R}^2(t_0') dx + 2\int_{-1}^{t_0'}\int_{B(x_0,2R)} \chi_{t_0,2R}^2 \varphi_{x_0,2R}^2|\nabla \hat{u}_{x_0,2R}|^2 dz\\ \leq \int_{-1}^{t_0'}\int_{B(x_0,2R)} |\hat{u}_{x_0,2R}|^2(\varphi_{x_0,2R}^2\partial_t \chi_{t_0,2R}^2 + \chi_{t_0,2R}^2 \Delta \varphi_{x_0,2R}^2)dz \\ -\int_{-1}^{t_0'}\int_{B(x_0,2R)}(F_{ik}-[F_{ik}]_{x_0,2R})(\varphi_{x_0,2R}^2)_{,ik} |\hat{u}_{x_0,2R}|^2 \chi_{t_0,2R}^2 dz\\ - 2 \int_{-1}^{t_0'}\int_{B(x_0,2R)}(F_{ik}-[F_{ik}]_{x_0,2R}) \chi_{t_0,2R}^2 (\varphi_{x_0,2R}^2)_{,i} \hat{u}^j_{x_0,2R}(\hat{u}^j_{x_0,2R})_k dz\\ - \int_{-1}^{t_0'} \chi_{t_0,2R}^2 u_{x_0,2R}\cdot\left( \int_{B(x_0,2R)} \hat{u}_{x_0,2R}\varphi_{x_0,2R}^2 \dvg \hat{u}_{x_0,2R} dx \right)dt, 
\end{multline}
for almost every $t_0' \in (-t_0 - (2R)^2,t_0)$, for all $z_0 = (x_0,t_0)\in Q$ and all $R>0$ satisfying the additional condition $Q(z_0,2R) \subset Q$.\\  
We will need some information on the regularity of $u_{x_0,2R}$ in the proof of \eqref{IV-E5}. What we can show is that 
\begin{equation}\label{IV-E6}
    \dot{u}_{x_0,2R}(:=\frac{d}{dt}u_{x_0,2R}) \in L_{\frac{3}{2}}(-1,0),
\end{equation}
and that's all we actually need to make our computations rigorous. To see this, we take as test function in \eqref{Toy-NS1},
\[
w^j_i (x,t) = \delta_{ij} \varphi^2_{x_0,2R}(x)\eta(t),
\]
where $\delta_{ij}$ is the Kronecker symbol and $\eta$ is an arbitrary function in $C^{\infty}_0(-1,0)$. As a result, we get that
\begin{multline}\label{IV-E7}
    \dot{u}^i_{x_0,2R}(t) = -\left(\int_{B(x_0,2R)} \nabla u_i \cdot \nabla \varphi^2_{x_0,2R}(x)dx + \int_{B(x_0,2R)}u\cdot \nabla u_i \varphi^2_{x_0,2R}(x)dx\right.\\ \left. + \frac{1}{2}\int_{B(x_0,2R)} u_i \varphi^2_{x_0,2R}(x) \dvg u dx\right) \left( \int_{B(x_0,2R)}\varphi^2_{x_0,2R}(x) dx \right)^{-1}, 
\end{multline}
which readily gives \eqref{IV-E6}.\\
Next, we replace $u(x,t)$, in his local energy inequality, by $\hat{u}_{x_0,2R}(x,t)+u_{x_0,2R}(t)$ and take as test function $\phi = \chi_{t_0,2R}^2 \varphi_{x_0,2R}^2$. Then, the terms that do not contain any spatial derivatives can be transform as follows
\begin{multline*}
    \int_{B(x_0,2R)}|\hat{u}_{x_0,2R}(x,t_0')+u_{x_0,2R}(t_0')|^2 \chi_{t_0,2R}^2(t_0') \varphi_{x_0,2R}^2(x) dx \\=  \int_{B(x_0,2R)}|\hat{u}_{x_0,2R}(x,t_0')|^2 \chi_{t_0,2R}^2(t_0') \varphi_{x_0,2R}^2(x) dx + \chi_{t_0,2R}^2(t_0')|u_{x_0,2R}(t_0')|^2\int_{B(x_0,2R)}\varphi_{x_0,2R}^2(x) dx,
\end{multline*}
and
\begin{multline*}
    \int_{-1}^{t_0'}\int_{B(x_0,2R)}|\hat{u}_{x_0,2R}+u_{x_0,2R}|^2 \partial_t\chi_{t_0,2R}^2 \varphi_{x_0,2R}^2 dz = \int_{-1}^{t_0'}\int_{B(x_0,2R)}|\hat{u}_{x_0,2R}|^2 \partial_t\chi_{t_0,2R}^2 \varphi_{x_0,2R}^2 dz\\ + \int_{B(x_0,2R)}\varphi_{x_0,2R}^2(x) dx \left( \chi_{t_0,2R}^2(t_0')|u_{x_0,2R}(t_0')|^2 - 2 \int_{-1}^{t_0'} u_{x_0,2R}(t)\cdot\dot{u}_{x_0,2R}(t)dt  \right).
\end{multline*}
Notice that 
\begin{multline*}
    \int_{-1}^{t_0'}\int_{B(x_0,2R)}|\hat{u}_{x_0,2R}+u_{x_0,2R}|^2 \chi_{t_0,2R}^2 \Delta\varphi_{x_0,2R}^2 dz = \int_{-1}^{t_0'}\int_{B(x_0,2R)}|\hat{u}_{x_0,2R}|^2 \chi_{t_0,2R}^2 \Delta\varphi_{x_0,2R}^2 dz\\ - 2 \int_{-1} \chi_{t_0,2R}^2 u_{x_0,2R} \cdot \left(\int_{B(x_0,R)} \nabla\hat{u}_{x_0,2R}\nabla \varphi_{x_0,2R}^2dx \right)dt.
\end{multline*}
Taking into account the previous expansions and \eqref{IV-E7}, the local energy becomes
\begin{multline}\label{IV-E8}
    \int_{B(x_0,2R)} |\hat{u}_{x_0,2R}(x,t_0')|^2 \varphi^2_{x_0,2R}(x)\chi_{t_0,2R}^2(t_0') dx + \int_{-1}^{t_0'}\int_{B(x_0,2R)} \chi_{t_0,2R}^2 \varphi_{x_0,2R}^2|\nabla \hat{u}_{x_0,2R}|^2 dz\\ \leq \int_{-1}^{t_0'}\int_{B(x_0,2R)} |\hat{u}_{x_0,2R}|^2(\varphi_{x_0,2R}^2\partial_t \chi_{t_0,2R}^2 + \chi_{t_0,2R}^2 \Delta \varphi_{x_0,2R}^2)dz + \int_{-1}^{t_0'}\int_{B(x_0,2R)} u\cdot \nabla \varphi_{x_0,2R}^2 |u|^2 \chi_{t_0,2R}^2 dz\\
    + 2 \int_{-1}^{t_0'}\chi_{t_0,2R}^2 u_{x_0,2R} \cdot \left( u\cdot \nabla u \varphi^2_{x_0,2R} + \frac{1}{2} u \varphi_{x_0,2R}^2 \dvg u dx\right)dt.
\end{multline}
All that is left is to transform the last two terms in \eqref{IV-E8} and notice that some cancellations occur when putting them together. We have
\begin{multline*}
    \int_{-1}^{t_0'}\int_{B(x_0,2R)} u\cdot \nabla \varphi_{x_0,2R}^2 |u|^2 \chi_{t_0,2R}^2 dz = A_0 + \int_{-1}^{t_0'}\int_{B(x_0,2R)} u\cdot \nabla \varphi^2_{x_0,2R} |\hat{u}_{x_0,2R}|^2 \chi^2_{t_0,2R}dx,
\end{multline*}
with
\[
A_0 = \int_{-1}^{t_0'} \chi^2_{t_0,2R} |u_{x_0,2R}|^2 \int_{B(x_0,2R)} u\cdot \nabla \varphi^2_{x_0,2R} dx dt + 2 \int_{-1} \chi^2_{t_0,2R}u_{x_0,2R} \cdot \left( u\cdot \nabla \varphi^2_{x_0,2R} \hat{u}_{x_0,2R}\right)dt. 
\]
Using the fact that $u = \dvg F$ and by integrating by part, we get
\begin{multline*}
    \int_{-1}^{t_0'}\int_{B(x_0,2R)} u\cdot \nabla \varphi^2_{x_0,2R} |\hat{u}_{x_0,2R}|^2 \chi^2_{t_0,2R}dx\\ = -\int_{-1}^{t_0'}\int_{B(x_0,2R)}(F_{ik}-[F_{ik}]_{x_0,2R})(\varphi_{x_0,2R}^2)_{,ik} |\hat{u}_{x_0,2R}|^2 \chi_{t_0,2R}^2 dz\\ - 2 \int_{-1}^{t_0'}\int_{B(x_0,2R)}(F_{ik}-[F_{ik}]_{x_0,2R}) \chi_{t_0,2R}^2 (\varphi_{x_0,2R}^2)_{,i} \hat{u}^j_{x_0,2R}(\hat{u}^j_{x_0,2R})_k dz.
\end{multline*}
Now, performing on more integration by part in the last term of \eqref{IV-E8}, we obtain that
\begin{align*}
    2 \int_{-1}^{t_0'}\chi_{t_0,2R}^2 u_{x_0,2R} \cdot \left( u\cdot \nabla u \varphi^2_{x_0,2R} + \frac{1}{2} u \varphi_{x_0,2R}^2 \dvg u dx\right)dt = -A_0\\ - \int_{-1}^{t_0'} \chi_{t_0,2R}^2 u_{x_0,2R}\cdot\left( \int_{B(x_0,2R)} \hat{u}_{x_0,2R}\varphi_{x_0,2R}^2 \dvg \hat{u}_{x_0,2R} dx \right)dt,
\end{align*}
which conclude the proof of \eqref{IV-E5} after putting all the terms together.
\paragraph{Step II.} We derive now a reverse holder inequality using the Caccioppoli's inequality established previously. Using simple known arguments, we get from \eqref{IV-E5} the following estimate
\begin{align*}
\phantom{{}\leq{}}
    I &:= \int_B |\hat{u}_{x_0,2R}(x,t_0)|^2 \varphi^2_{x_0,2R}(x)dx + 2 \int_{-1}^{t_0}\int_B\chi^2_{x_0,2R} \varphi^2_{x_0,2R}|\nabla \hat{u}_{x_0,2R}|^2 dz\\
    &\leq c \left( \frac{1}{R^2} \int_{Q(z_0,2R)}|\hat{u}_{x_0,2R}|^2 dz + \frac{1}{R^2}\int_{Q(z_0,2R)}|F-[F]_{x_0,2R}||\hat{u}_{x_0,2R}|^2 dz \right.\\&\mathrel{\phantom{=}}\left. \frac{1}{R}\int_{Q(z_0,2R)}|F-[F]_{x_0,2R}|(|\nabla \hat{u}_{x_0,2R}| \varphi_{x_0,2R}\chi_{t_0,2R})|\hat{u}_{x_0,2R}|dz\right.\\&\mathrel{\phantom{=}}\left.\int_{-1}^{t_0}|u_{x_0,2R}(t)|\int_B(|\nabla \hat{u}_{x_0,2R}| \varphi_{x_0,2R}\chi_{t_0,2R})|\hat{u}_{x_0,2R}|dz \right)\\
    &=: I_1 + I_2 + I_3 + I_4.
\end{align*}
Next, we estimate the $I_i$'s. For this we introduce $s\in (1,2)$ and obtain the following
\begin{align*}
    I_2 &\leq \int_{t_0 - (2R)^2}^{t_0} \left( \int_{B(x_0,2R)}|\hat{u}_{x_0,2R}|^{\frac{2 s}{2-s}} dx \right)^{\frac{2-s}{s}}\left( \int_{B(x_0,2R)}|F-[F]_{x_0,2R}|^{\frac{s}{2s-2}}dx\right)^{\frac{2s-2}{s}}dt\\
    &\leq C R^{2(\frac{3}{s'}-1)} \int_{t_0-(2R)^2}^{t_0}\left( \int_{B(x_0,2R)}|\hat{u}_{x_0,2R}|^{\frac{2 s}{2-s}} dx\right)^{\frac{2-s}{s}}dt,
\end{align*}
with $C = C(s,\Gamma)>0$ (where we set for simplicity $\Gamma:= \mbox{ess}\sup_{-1<t<0}\|F(\cdot,t)\|_{BMO(B)}$) and as usual $s' = s/(s-1)$. Similarly, we have 
\begin{align*}
    I_3 &\leq \frac{C(s)}{R}R^{\frac{3}{s'}} \mbox{ess}\sup_{t_0 - (2R)^2 < t < t_0} \sup_{B(x_0,2R) \subset B} \left( \frac{1}{|B(2R)|} \int_{B(x_0,2R)} |F-[F]_{x_0,2R}|^{s'}dx\right)^{\frac{1}{s'}}\\&\mathrel{\phantom{=}}\times \int_{t_0 - (2R)^2}^{t_0}\left( \int_{B(x_0,2R)} |\nabla \hat{u}_{x_0,2R}|^2 \varphi^2_{x_0,2R} \chi^2_{t_0,2R}dx\right)^{\frac{1}{2}}\left(\int_{B(x_0,2R)}|\hat{u}_{x_0,2R}|^{\frac{2s}{2-s}}dx\right)^{\frac{2-s}{2s}}dt\\
    &\leq C R^{\frac{3}{s'}-1} \left( \int_{Q(z_0,2R)}|\nabla \hat{u}_{x_0,2R}|^2 \varphi^2_{x_0,2R} \chi^2_{t_0,2R}dz\right)^{\frac{1}{2}} \left( \int_{t_0 - (2R)^2}^{t_0}\left(\int_{B(x_0,2R)}|\hat{u}_{x_0,2R}|^{\frac{2s}{2-s}}dx \right)^{\frac{2-s}{s}}dt \right)^{\frac{1}{2}},
\end{align*}
with $C = C(s,\Gamma)>0$. Next, notice that $u_{x_0,2R}(t)$ can be rewritten as follows
\[
u_{x_0,2R}(t) = -\frac{1}{R^3 \int_{B(2)}\varphi^2(x)dx} \int_{B(x_0,2R)}(F_{ik} - [F_{ik}]_{x_0,2R})(\varphi^2_{x_0,2R})_{,k} dx,
\]
thus,
\[ |u_{x_0,2R}(t)| \leq \frac{C}{R},\]
for all $t \in [-1,0]$ and $C= C(s,\varphi,\Gamma)>0$. We get, as before
\[
I_4 \leq C R^{\frac{3}{s'}-1} \left( \int_{Q(z_0,2R)}|\nabla \hat{u}_{x_0,2R}|^2 \varphi^2_{x_0,2R} \chi^2_{t_0,2R}dz\right)^{\frac{1}{2}} \left( \int_{t_0 - (2R)^2}^{t_0}\left(\int_{B(x_0,2R)}|\hat{u}_{x_0,2R}|^{\frac{2s}{2-s}}dx \right)^{\frac{2-s}{s}}dt \right)^{\frac{1}{2}},
\]
with $C = C(s,\varphi,\Gamma)>0$.\\
Summarising our efforts, we have
\begin{multline}\label{IV-E9}
    \int_B |\hat{u}_{x_0,2R}(x,t_0)|^2 \varphi^2_{x_0,2R}(x)dx + \int_{-1}^{t_0}\int_B\chi^2_{x_0,2R} \varphi^2_{x_0,2R}|\nabla \hat{u}_{x_0,2R}|^2 dz\\ \leq C(s,\varphi,\Gamma) R^{2(\frac{3}{s'}-1)} \int_{t_0-(2R)^2}^{t_0}\left( \int_{B(x_0,2R)}|\hat{u}_{x_0,2R}|^{\frac{2 s}{2-s}} dx\right)^{\frac{2-s}{s}}dt.
\end{multline}
It's worth mentioning that a careful analysis of the constant in the previous inequality shows the following dependency in $\Gamma$
\[ C = c(s,\varphi)(1 + \Gamma^2). \]
From this point the rest of the proof follows line by line the proof of a similar result in \cite{Ser12_2}. We, nonetheless, present the proof here, for the sake of completeness. Upon assuming $s \in (1,3/2)$, on can find without difficulty numbers $0<\lambda<1$, $0<\mu<1$ and $1<r<2$ such that 
\begin{gather*}
    \frac{2s}{2-s} = 2 \lambda + \frac{3 r}{3 - r} \mu,\\
    \lambda + \mu = 1,\\
    \frac{3 r}{3 - r}\mu \frac{2 - s}{s} = 1.
\end{gather*}
Using these numbers, we derive from \eqref{IV-E9}
\begin{align*}
    \int_{Q(z_0,R)}|\nabla \hat{u}_{x_0,2R}|^2 dz &\leq C R^{2(\frac{3}{s'}-1)}\int_{t_0-(2R)^2}^{t_0}\left( \int_{B(x_0,2R)}|\hat{u}_{x_0,2R}|^{2 \lambda + \frac{3r}{3-r}\mu} dx\right)^{\frac{2-s}{s}}dt\\
    &\leq C R^{2(\frac{3}{s'}-1)}\int_{t_0-(2R)^2}^{t_0} \left( \int_{B(x_0,2R)}|\hat{u}_{x_0,2R}|^2 dx \right)^{\frac{2-s}{s}\lambda}\left( \int_{B(x_0,2R)}|\hat{u}_{x_0,2R}|^{\frac{3r}{3-r}}dx \right)^{\frac{2-s}{s}\mu} dt.
\end{align*}
The last term can be estimated with the help of Sobolev's inequality
\begin{multline}\label{IV-E10}
    \int_{Q(z_0,R)}|\nabla \hat{u}_{x_0,2R}|^2 dz\leq C R^{2(\frac{3}{s'}-1)}\mbox{ess}\sup_{t_0 - (2R)^2< t<t_0} \left(\int_{B(x_0,2R)}|\hat{u}_{x_0,2R}(x,t)|^2 dx\right)^{\frac{1}{2}}\\
    \times R^{2\frac{r-1}{r}}\left( \int_{Q(z_0,2R)}|\nabla \hat{u}_{x_0,2R}|^r  \right)^{\frac{1}{r}},
\end{multline}
with $C = C(s,\varphi,\Gamma)>0$. To estimate the first multiplier on the right-hand side of the last inequality, we proceed as follows. By Poincare-Sobolev inequality, we have that 
\begin{equation}
    \left(\int_{B(x_0,2R)}|\hat{u}_{x_0,2R}|^{\frac{2 s}{2-s}} dx\right)^{\frac{2-s}{s}} \leq c(s) R^{3\frac{2-s}{s}-1}.
\end{equation}
Combining this with \eqref{IV-E9}, we get that 
\[
\int_B |\hat{u}_{x_0,2R}(x,t_0)|^2 \varphi^2_{x_0,2R}(x)dx \leq C \int_{Q(z_0,2R)}|\nabla \hat{u}_{x_0,2R}|^2 dx.
\]
Assuming now that $Q(z_0, 3R)\subset Q$, we have the following estimate
\begin{equation}\label{IV-E12}
    \mbox{ess}\sup_{t_0 - (2R)^2< t<t_0} \int_{B(x_0,2R)}|\hat{u}_{x_0,2R}(x,t)|^2 dx \leq C \int_{Q(z_0,3R)}|\nabla \hat{u}_{x_0,2R}|^2 dz,
\end{equation}
where $C = C(s,\varphi,\Gamma)>0$. We are now ready to estimate the fist multiplier on the right-hand side of \eqref{IV-E10}. We apply \eqref{IV-E12} in the following way
\begin{align*}
    \int_{B(x_0,2R)}|u(x,t) - u_{x_0,2R}(t)|^2 dx &\leq c \int_{B(x_0,2R)}|u(x,t) - u_{x_0,4R}(t)|^2 dx\\
    &\leq C \int_{Q(z_0,3R)}|\nabla \hat{u}_{x_0,2R}|^2 dz,
\end{align*}
for almost every $t \in (t_0 - (2 R)^2,t_0)$. Finally, \eqref{IV-E10} becomes
\begin{align*}
    \frac{1}{|Q(R)|}\int_{Q(z_0,R)}|\nabla u|^2 dz \leq C \left(\frac{1}{|Q(6R)|}\int_{Q(z_0,6R)}|\nabla u|^2 dz\right)^{\frac{1}{2}}\\
    \mathrel{\phantom{=}}\times \left(\frac{1}{|Q(2R)|}\int_{Q(z_0,2R)}|\nabla u|^r dz\right)^{\frac{1}{r}},
\end{align*}
which holds for some $r\in (1,2)$ and any $Q(z_0,6R)\subset Q$. And as before, $C = C(s,\varphi,\Gamma)>0$. This leads to (see \cite{Giaq82}) the existence of $p>2$ such that $\nabla u \in L_p(Q(R))$, for any $R\in (0,1)$. Moreover, the following estimate is valid
\begin{equation}
    \left(\frac{1}{|Q(R)|}\int_{Q(z_0,R)}|\nabla u|^p dz\right)^{\frac{1}{p}} \leq C \left(\frac{1}{|Q(6R)|}\int_{Q(z_0,6R)}|\nabla u|^2 dz\right)^{\frac{1}{2}},
\end{equation}
for all $Q(z_0,6R)\subset Q$ with $6R<$dist$(x_0,\partial B)$ and $t_0 - (6R)^2>-1$. Moreover, the constant $C>0$ depends only on $\Gamma$. This concludes the proof of Theorem \ref{Thm3.1}.
\section{Proof of Theorem~\ref{Thm3.2}}
We start by proving the following proposition.
\begin{proposition}\label{Prop5.1}
There exists an absolute positive constant $\epsilon_0$ with the following property. Assume that $u$ is suitable weak solution to \eqref{Toy-NS1} in $Q\equiv B\times (-1,0)$ and satisfies the condition
\begin{equation}\label{V-E14}
    \int_{Q} |u|^{\frac{10}{3}} dz < \epsilon_0.
\end{equation}
Then we have that $u$ is H\"older continuous in $\overline{Q(\varrho)}$ with $0<\varrho<1$.
\end{proposition}
Let us start with the proof of auxiliary lemmata and by mentioning that it is equivalent to prove Proposition \ref{Prop5.1} with condition \eqref{V-E14} replaced by
\begin{equation}\label{V-E15}
    \frac{1}{R}\left(\int_{Q(R)} |u|^{\frac{10}{3}} dz\right)^{\frac{3}{5}} < \epsilon_1,
\end{equation}
with $R$ fixed (say in $(0,1/2)$) and $u$ is now a suitable weak solution in $Q(2 R)$.
We also introduce the following notations
\[
\begin{gathered}
M(z_0,R) := \frac{1}{R}\left(\int_{Q(z_0,R)} |u|^{\frac{10}{3}} dz\right)^{\frac{3}{5}}\\
|u|^2_{2,Q(z_0,R)} := \mbox{ess}\sup_{t_0-R^2< t < t_0} \int_{B(x_0,R)}|u(\cdot,t)|^2 dx + \int_{Q(z_0,R)}|\nabla u|^2 dz
\end{gathered}
\]
and for simplicity, we take $M(R)= M(0,R)$. Also, unless otherwise specified all the constants $c$ in this section are positive universal constants.
\begin{lemma}[A Caccioppoli type inequality]\label{V-L1}
Let $u$ be a suitable weak solution to \eqref{Toy-NS1} in $Q(2R)$ and $\tau \in (0,1)$, then
\begin{equation}
    |\bar{u}|^2_{2,Q(\tau R)} \leq \frac{c}{ R}\left( \int_{Q(R)} |\bar{u}|^{\frac{5}{2}}dz\right)^{\frac{4}{5}}\left(\frac{1+M(R)}{(1-\tau)^2}+\frac{M(R)}{\tau^3}\right),
\end{equation}
where $\bar{u} := u - (u)_{,\tau R}$.
\end{lemma}
\begin{proof}
Making use of the fact that $u$ is a suitable weak solution, we have that
\begin{multline*}
    \int_{B(R)}|\bar{u}(\cdot,t_0)|^2\varphi^2_{\varrho,r}(\cdot,t_0) dx + 2 \int_{-R^2}^{t_0}\int_{B(R)}|\nabla \bar{u}|^2\varphi^2_{\varrho,r} dz\\ \leq \int_{-R^2}^{t_0}\int_{B(R)}|\bar{u}|^2(\partial_t \varphi^2_{\varrho,r} + \Delta \varphi^2_{\varrho,r}) dz + \int_{-R^2}^{t_0}\int_{B(R)}|\bar{u}|^2 u\cdot \nabla \varphi^2_{\varrho,r} dz \\ + (u)_{,r}\cdot \int_{-R^2}^{t_0}\int_{B(R)} \bar{u}\varphi^2_{\varrho,r} (\dvg u) dz,
\end{multline*}
for a.e. $t_0 \in (-R^2,0)$, $0<r<\varrho\leq R$; here $0\leq\varphi_{\varrho,r}\leq 1$ is a cut-off function with the following properties: $\varphi_{\varrho,r} \in C^{\infty}_0(B(\varrho)\times (-\varrho^2,\varrho))$, $\varphi_{\varrho,r}\equiv 1$ in $B(r)\times (-r^2,r^2)$, $|\nabla^k \varphi_{\varrho,r}| \leq c/(\varrho-r)^k$, $k=1,2$, and $\partial_t \varphi_{\varrho,r} \leq c/(\varrho - r)^2$. From the previous inequality, we get that
\begin{multline}\label{V-E17}
    |\bar{u}\varphi_{\varrho,r}|^2_{2,Q(\varrho)} \leq c \left[\left(\frac{1}{(\varrho - r)^2}+|(u)_{,r}|^2\right)\int_{B(\varrho)}|\bar{u}|^2 dz\right.\\ +\left. \frac{1}{\varrho-r}\left(\int_{B(\varrho)}|\bar{u}|^{\frac{5}{2}}dz\right)^{\frac{2}{5}}\left(\int_{Q(\varrho)}|\bar{u}\varphi_{\varrho,r}|^{\frac{10}{3}}dz \right)^{\frac{3}{10}}
    \left(\int_{Q(\varrho)}|u|^{\frac{10}{3}}dz \right)^{\frac{3}{10}} \right].
\end{multline}
By interpolation inequality and Sobolev embedding, we have that
\[
\|\bar{u}\varphi_{\varrho,r}\|_{\frac{10}{3},Q(\varrho)} \leq c |\bar{u}\varphi_{\varrho,r}|_{2,Q(\varrho)}.
\]
Therefore, using Young's inequality for the last term on the right hand side of \eqref{V-E17}, we have
\[ 
|\bar{u}\varphi_{\varrho,r}|^2_{2,Q(\varrho)} \leq c\left[ \left(\frac{1}{(\varrho - r)^2}+|(u)_{,r}|^2\right)\int_{B(\varrho)}|\bar{u}|^2 dz + \frac{R M(R)}{(\varrho-r)^2} \left(\int_{B(\varrho)}|\bar{u}|^{\frac{5}{2}}dz\right)^{\frac{4}{5}}\right].
\]
Next, using the fact that
\[
\int_{B(\varrho)}|\bar{u}|^2 dz \leq |Q(R)|^{\frac{1}{5}}\left(\int_{B(\varrho)}|\bar{u}|^{\frac{5}{2}}dz\right)^{\frac{4}{5}},
\]
and that
\[
|(u)_{,r}|^2 \leq \frac{c}{r^3}RM(R),
\]
we get
\begin{equation}\label{V-E18}
    |\bar{u}\varphi_{\varrho,r}|^2_{2,Q(\varrho)} \leq c \left( \frac{R+RM(R)}{(\varrho-r)^2} + \frac{R^2 M(R)}{r^3} \right)\left(\int_{B(\varrho)}|\bar{u}|^{\frac{5}{2}}dz\right)^{\frac{4}{5}}.
\end{equation}
Finally, taking $r = \tau R$, $\varrho = R$, we have that the lemma is proved.
\end{proof}
We consider now the following initial boundary value problem
\begin{equation}\label{V-E19}
    \begin{cases}
    \partial_t w - \Delta w = F\quad \mbox{in }Q(\frac{3}{4}R)\\
    w|_{\partial'Q(\frac{3}{4}R)} = 0,
    \end{cases}
\end{equation}
where $F := -u\cdot \nabla u - \frac{1}{2}u\dvg u$ and the symbol "$\partial'$" stands for the parabolic boundary; we have on one hand that 
\begin{equation}\label{V-E20}
    \int_{Q(\frac{3}{4}R)}|F|^{\frac{5}{4}}dz \leq c  \left(|\bar{u}|^2_{2,Q(\frac{3}{4}R)}RM(R) \right)^{\frac{5}{8}}.
\end{equation}
On the other hand, we have that the problem \eqref{V-E19} is uniquely solvable and moreover and the following estimate holds
\begin{equation}\label{V-E21}
    \int_{Q(\frac{3}{4}R)}\left(|\partial_t w|^{\frac{5}{4}} + |\nabla^2 w|^{\frac{5}{4}}\right)dz \leq c \int_{Q(\frac{3}{4}R)}|F|^{\frac{5}{4}} dz.
\end{equation}
Next, we have thanks to parabolic embeddings that
\begin{equation}\label{V-E22}
    \int_{Q(\frac{3}{4}R)}|w|^{\frac{5}{2}}dz \leq c\left[ \int_{Q(\frac{3}{4}R)}\left(|\partial_t w|^{\frac{5}{4}} + |\nabla^2 w|^{\frac{5}{4}}\right)dz\right]^2.
\end{equation}
Finally, by combining inequalities \eqref{V-E20},\eqref{V-E21}, \eqref{V-E22} and using Lemma \ref{V-L1} (with $\tau = 3/4$), we get that
\begin{equation}\label{V-E23}
    \int_{Q(\frac{3}{4}R)}|w|^{\frac{5}{2}}dz \leq c [(1+M(R))M(R)]^{\frac{5}{4}}\int_{Q(\frac{3}{4}R)}|\bar{u}|^{\frac{5}{2}}dz.
\end{equation}
Now, we introduce the function $v:= u-w$ and notice that
\[
\partial_t v - \Delta v = 0, \quad\mbox{in } Q(\frac{3}{4}R),
\]
and therefore, the following estimate is valid
\begin{equation}\label{V-E24}
    \int_{Q(r)}|v -(v)_{,r}|^{\frac{5}{2}} dz \leq c \left( \frac{r}{\varrho}\right)^{5+2}\int_{Q(\varrho)}|v -(v)_{,\varrho}|^{\frac{5}{2}} dz,
\end{equation}
for all $0<r<\varrho\leq 3R/4$.
Next, we have the following lemma
\begin{lemma}\label{V-L2}
Let $u$ be a suitable weak solution to \eqref{Toy-NS1} in $Q(2R)$, then
\[
\int_{Q(r)}|u -(u)_{,r}|^{\frac{5}{2}} dz \leq c\left[\left(\frac{r}{R}\right)^{5+2} + (1+M(R))^\frac{5}{4}M(R)^{\frac{5}{4}} \right]\int_{Q(\frac{3}{4}R)}|u -(u)_{,\frac{3}{4}R}|^{\frac{5}{2}} dz,
\]
for all $0<r<3R/4$.
\end{lemma}
\begin{proof}
 We have, for all $0<r<3R/4$, that
 \begin{align*}
     \phantom{{}\leq{}}
     \int_{Q(r)}|u -(u)_{,r}|^{\frac{5}{2}} dz &\leq c\left( \int_{Q(r)}|v -(v)_{,r}|^{\frac{5}{2}} dz + \int_{Q(r)}|w -(w)_{,r}|^{\frac{5}{2}} dz\right)\\
     &\leq c\left( \frac{r}{R}\right)^{5+2}\int_{Q(\frac{3}{4}R)}|v -(v)_{,\frac{3}{4}R}|^{\frac{5}{2}} dz + c \int_{Q(r)}|w|^{\frac{5}{2}}dz \mbox{ (here we used \eqref{V-E24})}\\
     &\leq c\left( \frac{r}{R}\right)^{5+2}\int_{Q(\frac{3}{4}R)}|u -(u)_{,\frac{3}{4}R}|^{\frac{5}{2}} dz + c \int_{Q(\frac{3}{4}R)}|w|^{\frac{5}{2}}dz\\
     &\leq c\left[\left(\frac{r}{R}\right)^{5+2} + (1+M(R))^\frac{5}{4}M(R)^{\frac{5}{4}} \right]\int_{Q(\frac{3}{4}R)}|u -(u)_{,\frac{3}{4}R}|^{\frac{5}{2}} dz,
 \end{align*}
 where \eqref{V-E23} is used to obtain the last line.
\end{proof}
Our goal now is to iterate Lemma \ref{V-L2} (see \cite{Ser03} for a similar situation). We start with the following lemma.
\begin{lemma}\label{V-L3}
Let $u$ be a suitable weak solution to \eqref{Toy-NS1} in $Q(2R)$ and $\tau \in (0,1)$, then
\[
\sqrt{M(\tau^{k+1} R)} \leq \frac{c}{(1-\tau)\tau^{\frac{7}{2}}}\sum_{i=0}^{k-1}\tau^i(1+M(\tau^i R))^{\frac{1}{2}}M(\tau^i R)^{\frac{1}{2}}
    + \tau^k \sqrt{M(\tau R)},
\]
with $k=1,2,\ldots$
\end{lemma}
\begin{proof}
 We have 
 \begin{align*}
     \sqrt{\tau^2 R M(\tau^2 R)} & \leq \left( \|u-(u)_{\tau R}\|_{10/3,Q(\tau R)} + |(u)_{\tau R}| \times |Q(\tau^2 R )|^{\frac{3}{10}}\right)\\
     & \leq c |u - (u)_{\tau R}|_{2,Q(\tau R)} + \tau^{\frac{3}{2}}\sqrt{\tau R M(\tau R)},
 \end{align*}
 thus making use of Lemma \ref{V-L1}, we get
\begin{equation}\label{V-E25}
    \sqrt{M(\tau^2 R)} \leq \frac{c}{(1-\tau)\tau^{\frac{5}{2}}}(1+M(R))^{\frac{1}{2}}\left(\frac{1}{R^{\frac{5}{2}}}\int_{Q(R)}|u-(u)_{,\tau R}|^{\frac{5}{2}}dz \right)^{\frac{2}{5}}\\
    + \tau \sqrt{M(\tau R)}.
\end{equation}
Now notice that (since $\tau<1$)
\[
\int_{Q(R)}|u-(u)_{,\tau R}|^\frac{5}{2}dz \leq \frac{c}{\tau^5}\int_{Q(R)}|u-(u)_{,R}|^\frac{5}{2}dz,
\]
therefore \eqref{V-E25} becomes
\begin{equation}\label{V-E26}
    \sqrt{M(\tau^2 R)} \leq \frac{c}{(1-\tau)\tau^{\frac{7}{2}}}(1+M(R))^{\frac{1}{2}}\left(\frac{1}{R^{\frac{5}{2}}}\int_{Q(R)}|u-(u)_{,R}|^{\frac{5}{2}}dz \right)^{\frac{2}{5}}\\
    + \tau \sqrt{M(\tau R)};
\end{equation}
and the lemma is proved by iterating the above inequality and noticing that 
\begin{equation}
    \left(\frac{1}{(\tau^i R)^{\frac{5}{2}}}\int_{Q(\tau^i R)}|u-(u)_{,\tau^i R}|^{\frac{5}{2}}dz\right)^{\frac{2}{5}} \leq c M(\tau^i R)^{\frac{1}{2}},
\end{equation}
for every integer $i$.
\end{proof}
\begin{remark}\label{V-R1}
    Let us notice that for $\epsilon_1 \in (0,1)$ as given in \eqref{V-E15} and for any $\tau \in (0,1/4)$, we have that
    \begin{equation}\label{V-E28}
        M(\tau^k R) \leq \frac{\epsilon_1}{B},
    \end{equation}
    for every positive integer, and with
    \[ B = B(\tau) = \max\left\{ \frac{(1-\tau)^2\tau^7}{2^7 c^2}, \frac{3 \tau}{1-4\tau} \right\}, \]
    here the constant $c$ is the same as in Lemma \ref{V-L3}. We can also show without too much difficulty that 
    \[ 
    (1+\frac{\epsilon_1}{B})\times\frac{\epsilon_1}{B} \leq \frac{(1-\tau)}{3\tau}\times \frac{\epsilon_1}{B};
    \]
    set for simplicity
    \[ B_1 = \frac{3 \tau B}{1-\tau} \]
\end{remark}
By iterating Lemma \ref{V-L2} and taking into account Remark \ref{V-R1}, we obtain the following.
\begin{lemma}\label{V-L4}
Let $u$ be a suitable weak solution to \eqref{Toy-NS1} in $Q(2R)$ such that \eqref{V-E15} holds, and let $\tau \in (0,1/4)$; then
\[
\int_{Q(r)}|u -(u)_{,r}|^{\frac{5}{2}} dz \leq \frac{2^{\frac{5}{2}}}{\tau^{12}R^6}r^{5+2-1}\int_{Q(\tau R)}|u -(u)_{,\tau R}|^{\frac{5}{2}} dz,
\]
for all $0<r<\tau R$ whenever $\tau \leq 1/(2 c)$, with $c$ given in \eqref{V-E29} and $\ee_1$ is choosen such that $(\epsilon_1/B_1)^{\frac{5}{4}} \leq \tau^{5+2}$.
\end{lemma}
\begin{proof}
 We obviously have from Lemma \ref{V-L2} that
 \[
\int_{Q(r)}|u -(u)_{,r}|^{\frac{5}{2}} dz \leq c\left[\left(\frac{r}{R}\right)^{5+2} + (1+M(R))^\frac{5}{4}M(R)^{\frac{5}{4}} \right]\int_{Q(R)}|u -(u)_{,R}|^{\frac{5}{2}} dz,
\]
for all $0<r<3R/4$. We take $r=\tau R$ and derive the following recursive formula
\begin{equation*}
    \int_{Q(\tau^{k+1} R)}|u -(u)_{,\tau^{k+1}R}|^{\frac{5}{2}} dz \leq c\left[\tau^{5+2} + (1+M(\tau^k R))^\frac{5}{4}M(\tau^k R)^{\frac{5}{4}} \right]\int_{Q(\tau^k R)}|u -(u)_{,\tau^k R}|^{\frac{5}{2}} dz,
\end{equation*}
Setting for simplicity
\[ \Phi(r) := \int_{Q(r)}|u -(u)_{,r}|^{\frac{5}{2}} dz, \]
and taking into account Remark \ref{V-R1}, we get
\begin{equation}\label{V-E29}
    \Phi(\tau^{k+1} R) \leq c (\tau^{5+2} + (\epsilon_1/B_1)^{5/4})\Phi(\tau^k R).
\end{equation}
We add the following additional restriction $\tau\leq \min\{1/(2c),1/4\}$ ($c$ as in \eqref{V-E29}) and define $\epsilon_* = \tau^{5+2}$; we have for 
\begin{equation}
    \left(\frac{\epsilon_1}{B_1}\right)^{\frac{5}{4}} \leq \epsilon_*,
\end{equation}
that
\begin{align*}
    \Phi(\tau^{k+1} R) &\leq c \tau \tau^{5+2-1}(1+\epsilon_*\tau^{-5-2})\Phi(\tau^k R)\\
    &\leq \tau^{5+2-1}\Phi(\tau^k R).
\end{align*}
Iterating the last inequality in $k$ starting with $k=1$, we find
\[ \Phi(\tau^k R) \leq \tau^{(k-1)(5+2-1)}\Phi(\tau R), \]
for any positive integer $k$. The remaining of the proof is fairly standard. The lemma is proved.
\end{proof}
We are now ready to prove Proposition \ref{Prop5.1}.
\begin{proof}[Proof of Proposition \ref{Prop5.1}]
Clearly, there exists $0<\varrho<R/8$, such that 
 \[ M(z_0,R) < \epsilon_1, \]
 for all $z_0\in Q(\varrho)$, with the same $R$ as in \eqref{V-E15}. Consequently, Lemma \ref{V-L4} can be strengthen as follows (we just repeat the above argument with $Q(z_0,R)$ instead of $Q(R)$)
 \begin{align*}
     \int_{Q(z_0,r)}|u -(u)_{z_0,r}|^{\frac{5}{2}} dz &\leq \frac{2^{\frac{5}{2}}}{\tau^{12}R^6}r^{5+2-1}\int_{Q(z_0,\tau R)}|u -(u)_{z_0,\tau R}|^{\frac{5}{2}} dz,\\
     &\leq \frac{2^{\frac{7}{2}}}{\tau^{12}R^6}r^{5+2-1}\int_{Q(2 R)}|u -(u)_{,2 R}|^{\frac{5}{2}} dz,
 \end{align*}
 for all $z_0\in Q(\varrho)$ and $\tau\leq \min\{1/(2c),1/4\}$. The conclusion follows from Campanato's type condition for parabolic H\"older continuity. Proposition \ref{Prop5.1} is then proved.  
\end{proof}
As a straightforward consequence of Proposition \ref{Prop5.1}, we have the following.
\begin{proposition}\label{Prop5.2}
There exists an absolute positive constant $\epsilon_0$ with the following property. Assume that $u$ is suitable weak solution to \eqref{Toy-NS1} in $Q\equiv B\times (-1,0)$ and satisfies the condition
\begin{equation}\label{V-E31}
    \int_{Q} |u|^3 dz < \epsilon_0.
\end{equation}
Then we have that $u$ is H\"older continuous in $\overline{Q(\varrho)}$ with $0<\varrho<1$.
\end{proposition}
\begin{proof}
 The proof is an easy consequence of the following estimate
 \[
 \|u\|^2_{10/3,Q(1/2)} \leq c|u|^2_{2,Q(1/2)} \leq c \left[ \int_Q|u|^3dz + \left( \int_Q|u|^3 \right)^{2/3} \right],
 \]
 and Proposition \ref{Prop5.1}.
\end{proof}
We turn now to the proof of Theorem \ref{Thm3.2}. The following scaled energy quantities will be needed
\begin{equation}\label{E3.39}
\begin{gathered}
A(r) := \sup_{t_0-r^2\leq t \leq t_0}\frac{1}{r} \int_{B(x_0,r)}|u(x,t)|^2 dx, \quad E(r):= \frac{1}{r}\int_{Q(z_0,r)}|\nabla u|^2 dz\\
C(r) := \frac{1}{r^2}\int_{Q(z_0,r)}|u|^3 dz
\end{gathered}
\end{equation}
Let us start first proving some auxiliary lemmata.
\begin{lemma}\label{V-L5}
For all $0<r\leq \varrho <1$,
\[ C(r) \leq c\left[ \left( \frac{r}{\varrho}\right)^3A^{3/2}(\varrho) + \left( \frac{\varrho}{r}\right)^3A^{3/4}(\varrho)E^{3/4}(\varrho) \right]. \]
\end{lemma}
\begin{proof}
We have
\begin{align*}
\int_{B(r)}|u|^2 dx &= \int_{B(r)}\left( |u|^2 - [|u|^2]_{,\varrho} \right)dx + \left(\frac{r}{\varrho}\right)^3\int_{B(\varrho)}|u|^2 dx\\
&\leq \int_{B(\varrho)}\left| |u|^2 - [|u|^2]_{,\varrho} \right|dx + \left(\frac{r}{\varrho}\right)^3\int_{B(\varrho)}|u|^2 dx.
\end{align*}
By Poincar\'e's inequality on the ball, we have 
\[ \int_{B(\varrho)}\left| |u|^2 - [|u|^2]_{,\varrho} \right|dx \leq c \int_{B(\varrho)} |\nabla u| |u| dx, \]
(where $c$, as usual, is an absolute positive constant). Therefore, we get
\begin{equation}\label{V-E33}
\begin{split}
\int_{B(r)}|u|^2 dx &\leq c\varrho\left( \int_{B(\varrho)} |\nabla u|^2dx \right)^{1/2}\left( \int_{B(\varrho)}|u|^2\right)^{1/2} + \left(\frac{r}{\varrho}\right)^3\int_{B(\varrho)}|u|^2 dx\\
&\leq c\varrho^{3/2}A^{1/2}(\varrho)\left( \int_{B(\varrho)} |\nabla u|^2dx \right)^{1/2} + \left(\frac{r}{\varrho}\right)^3 \varrho A(\varrho).
\end{split}
\end{equation}
By interpolation inequality (and Sobolev embedding and Poincar\'e's inequality on the ball), we obtain that
\begin{align*}
\int_{B(r)}|u|^3 dx &\leq c\left[ \left(\int_{B(r)}|\nabla u|^2 dx\right)^{3/4}\left(\int_{B(r)}|u|^2\right)^{3/4} + \frac{1}{r^{3/2}}\left(\int_{B(r)}|u|^2 dx\right)^{3/2} \right]\\
&\leq c\left\{ \varrho^{3/4}A^{3/4}(\varrho)\left(\int_{B(r)}|\nabla u|^2 dx\right)^{3/4} + \frac{1}{r^{3/2}}\left[ c\varrho^{3/2}A^{1/2}(\varrho)\left( \int_{B(\varrho)} |\nabla u|^2dx \right)^{1/2}+\right.\right.\\ &\mathrel{\phantom{=}} + \left.\left. \left(\frac{r}{\varrho}\right)^3 \varrho A(\varrho) \right]^{3/2} \right\}\\
&\leq c\left\{ \left(\frac{r}{\varrho}\right)^3A^{3/2}(\varrho) + \left( \int_{B(\varrho)} |\nabla u|^2dx \right)^{3/4}\left[\varrho^{3/4}+\frac{\varrho^{9/4}}{r^{3/2}}\right]A^{3/4}(\varrho) \right\}.
\end{align*}
Integrating the latter inequality in $t$ on $(t_0-r^2,t_0)$, we get 
\begin{align*}
\int_{Q(r)}|u|^3 dz &\leq c \left\{ r^2\left(\frac{r}{\varrho}\right)^3A^{3/2}(\varrho) + \left[\varrho^{3/4}+\frac{\varrho^{9/4}}{r^{3/2}}\right]A^{3/4}(\varrho)r^{1/2}\left( \int_{Q(\varrho)} |\nabla u|^2dx \right)^{3/4}  \right\}\\
&\leq c \left\{ r^2\left(\frac{r}{\varrho}\right)^3A^{3/2}(\varrho) + \left[\varrho^{3/4}+\frac{\varrho^{9/4}}{r^{3/2}}\right]A^{3/4}(\varrho)r^{1/2}E^{3/4}(\varrho)\varrho^{3/4}  \right\}.
\end{align*}
Noticing that 
\[ \left[\varrho^{3/4}+\frac{\varrho^{9/4}}{r^{3/2}}\right]r^{1/2}\varrho^{3/4} = \left[ \left(\frac{\varrho}{r}\right)^{3/2}+\left(\frac{\varrho}{r}\right)^3 \right]r^2 \leq 2\left(\frac{\varrho}{r}\right)^3r^2, \]
we have that the proof is completed.
\end{proof}
\begin{lemma}\label{V-L6}
For any $0<R<1$, 
\[ A(R/2) + E(R/2) \leq c \left[ C^{2/3}(R) + A^{1/2}(R)C^{1/3}(R)E^{1/2}(R) \right] \].
\end{lemma}
\begin{proof}
Picking up a suitable cut-off function in the energy inequality (see Definition \ref{def2.1}), we get the following estimates
\begin{multline}
A(R/2) + E(R/2) \leq c\left[ \frac{1}{R^3} \int_{Q(R)}|u|^2 dz + \frac{1}{R^2}\int_{Q(R)} \left| |u|^2-[|u|^2]_{,R}\right||u|dz\right.\\ + \left.\int_{-R^2}^0 [|u|^2]_{,R}\int_{B(R)}\frac{1}{R}|\nabla u|dx dt. \right]
\end{multline}
First, let us notice that 
\[ \frac{1}{R^3}\int_{Q(R)}|u|^2 dz \leq c C^{2/3}(R); \]
next, 
\begin{align*}
\int_{-R^2}^0 [|u|^2]_{,R}\int_{B(R)}\frac{1}{R}|\nabla u|dx dt &= c \int_{-R^2}^0 \left(\frac{1}{R^3}\int_{B(R)}|u|^2dx\right)^{1/2}\left(\frac{1}{R^3}\int_{B(R)}|u|^2dx\right)^{1/2} \left( \frac{1}{R}\int_{B(R)}|\nabla u|dx\right) dt\\
&\leq c\frac{A(R)^{\frac{1}{2}}}{R}\left(\frac{1}{R^3}\int_{Q(R)}|u|^2dz\right)^{1/2}\left(\int_{-R^2}^0\left( \frac{1}{R}\int_{B(R)}|\nabla u|dx\right)^2 dt\right)^{1/2}\\
&\leq c\frac{A(R)^{\frac{1}{2}}}{R} C^{1/3}(R) R E^{1/2}(R)\\
&\leq c A(R)^{1/2}C^{1/3}(R) E^{1/2}(R).
\end{align*}
We dealt with the last term as follows
\begin{align*}
\int_{Q(R)}\left| |u|^2-[|u|^2]_{,R}\right||u|dz &\leq \int_{-R^2}^0 \left(\int_{B(R)}\left| |u|^2-[|u|^2]_{,R}\right|^{3/2}\right)^{2/3}\left( \int_{B(R)}|u|^3\right)^{1/3}\\
&\leq c\int_{-R^2}^0 \left( \int_{B(R)}|\nabla u|^2dx\right)^{1/2}\left( \int_{B(R)}|u|^2 dx\right)^{1/2}\left( \int_{B(R)}|u|^3dx\right)^{1/3} dt\\
&\leq c R^{1/2}A^{1/2}(R) \left( \int_{Q(R)}|u|^3dz\right)^{1/3}\left(\int_{-R^2}^0\left( \int_{B(R)}|\nabla u|^2 dx\right)^{3/4}dt\right)^{2/3}\\
&\leq c R^{1/2+2/3}A^{1/2}(R)C^{1/3}(R)R^{1/3}\left( \int_{Q(R)}|\nabla u|^2dz\right)^{1/2}\\
&\leq c R^2 A^{1/2}(R)C^{1/3}(R) E^{1/2}(R),
\end{align*}
which concludes the proof.
\end{proof}
\begin{proof}[Proof of Theorem~\ref{Thm3.2}]
It follows from Lemma~\ref{V-L5} and the assumptions of Theorem~\ref{Thm3.2} that:
\begin{equation}\label{V-E35}
C(r) \leq c \left[ \left(\frac{\varrho}{r}\right)^3 A^{3/4}(\varrho)\epsilon^{3/4} + \left(\frac{r}{\varrho}\right)^3 A^{3/2}(\varrho) \right]
\end{equation}
Introducing, the new quantity 
\[ \mathcal{E}(r) := A^{3/2}(r),\]
we derive from Lemma~\ref{V-L6}
\begin{equation}\label{V-E36}
\mathcal{E}(r) \leq \left[ C(2r) + A^{3/4}(2r)C^{1/2}(2r)\epsilon^{3/4} \right].
\end{equation}
Now let us assume that $0<r\leq \varrho/2<\varrho\leq 1$. Replacing $r$ with $2r$ in \eqref{V-E35}, we can reduce \eqref{V-E36} to the form
\begin{align*}
\mathcal{E}(r) &\leq c\left[ \left(\frac{\varrho}{r}\right)^3 A^{3/4}(\varrho)\epsilon^{3/4} + \left(\frac{r}{\varrho}\right)^3 A^{3/2}(\varrho) + A^{3/4}(2r)\left( \left(\frac{\varrho}{r}\right)^3 A^{3/4}(\varrho)\epsilon^{3/4} + \left(\frac{r}{\varrho}\right)^3 A^{3/2}(\varrho) \right)^{1/2}\epsilon^{3/4} \right]\\
&\leq c\left[ \left(\frac{\varrho}{r}\right)^3 A^{3/4}(\varrho)\epsilon^{3/4} + \left(\frac{r}{\varrho}\right)^3 A^{3/2}(\varrho) + \left(\frac{\varrho}{r}\right)^{3/2+3/4}A^{3/4+3/8}(\varrho)\epsilon^{3/4+3/8}\right.\\&\mathrel{\phantom{=}} + \left.\left(\frac{\varrho}{r}\right)^{3/4}A^{3/4}(\varrho)\left(\frac{r}{\varrho}\right)^{3/2}A^{3/4}(\varrho)\epsilon^{3/4}\right].
\end{align*}
Here, the obvious inequality $A(2 r) \leq c\varrho A(\varrho)/r$ has been used. Applying Young's inequality with an arbitrary positive constant $\delta$, we show that 
\[ \mathcal{E}(r) \leq c\left(\frac{r}{\varrho}\right)^{3/4}(\epsilon^{3/4}+1)\mathcal{E}(\varrho) + c\delta \mathcal{E}(\varrho) + c(\delta)\left( \epsilon^{3/2}\left(\frac{\varrho}{r}\right)^6 + \left(\frac{\varrho}{r}\right)^9\epsilon^{9/2}\right). \]
Therefore,
\begin{equation}\label{V-E37}
\mathcal{E}(r) \leq c \left[ \left(\frac{r}{\varrho}\right)^{3/4}(\epsilon^{3/4}+1) + \delta \right]\mathcal{E}(\varrho) + c(\delta)\left(\frac{\varrho}{r}\right)^9(\epsilon^{3/2}+\epsilon^{9/2}).
\end{equation}
Inequality \eqref{V-E37} holds for $r\leq \varrho/2$ and can be rewritten as follows:
\begin{equation}\label{V-E38}
\mathcal{E}(\vartheta \varrho) \leq c \left[ \vartheta^{3/4}(\epsilon^{3/4}+1) + \delta \right]\mathcal{E}(\varrho) + c(\delta)\vartheta^{-9}(\epsilon^{3/2}+\epsilon^{9/2})
\end{equation}
for any $0<\vartheta\leq 1/2$ and any $0<\varrho\leq 1$.\\
Now, assuming that $\epsilon \leq 1$, let us fix $\vartheta$ and $\delta$ such that 
\begin{equation}\label{V-E39}
2c\vartheta^{1/4} < 1/2,\quad 0<\vartheta \leq 1/2,\quad c\delta < \vartheta^{1/2}/2.
\end{equation}
Obviously, $\vartheta$ and $\delta$ are independent of $\epsilon$. So,
\begin{equation}\label{V-E40}
\mathcal{E}(\vartheta \varrho) \leq \vartheta^{1/2}\mathcal{E}(\varrho) + G
\end{equation}
for any $0<\varrho\leq 1$, where $G=G(\epsilon) \to 0$ as $\epsilon\to 0^+$.\\
Iterating \eqref{V-E40}, we obtain
\begin{equation}\label{V-E41}
\mathcal{E}(\vartheta^k \varrho) \leq \vartheta^{k/2}\mathcal{E}(\varrho) + cG,
\end{equation}
for any natural number $k$ and any $0<\varrho\leq 1$. Letting $\varrho = 1$, we obtain that
\begin{equation}\label{V-E42}
\mathcal{E}(\vartheta^k) \leq \vartheta^{k/2}\mathcal{E}(1) + cG,
\end{equation}
for the same values of $k$. Hence, it can be easily deduced from \eqref{V-E42}, that 
\begin{equation}\label{V-E43}
\mathcal{E}(r) \leq c\left(r^{1/2}\mathcal{E}(1) + G(\epsilon)\right),
\end{equation}
for all $0<r\leq 1/2$. Now, \eqref{V-E35} implies, for $0<r\leq 1/4$
\begin{align*}
C(r) &\leq c \left[\mathcal{E}^{1/2}(2r)\epsilon^{3/4} + \mathcal{E}(2r)\right]\\
&\leq c\left[ r^{1/4}\mathcal{E}^{1/2}\epsilon^{3/4} + G(\epsilon)^{1/2}\epsilon^{3/4} + r^{1/2}\mathcal{E}(1) + G(\epsilon) \right].
\end{align*}
Choosing $\epsilon$ sufficiently small and $r_0$ also sufficiently small, we obtain that
\[ C(r_0) < \epsilon_0, \]
where $\epsilon_0$ is as in Proposition \ref{Prop5.2}. Since $u$ is suitable weak solution in $Q(r_0)$, Proposition \ref{Prop5.2} and the scaling symmetry of our system yield the required statement and the estimate holds for the case $k=l=0$; the other cases can be obtained by a straightforward bootstrap argument. Thus, Theorem~\ref{Thm3.2} is proved.
\end{proof}
\section{Proof of Theorem~\ref{Thm3.3}}
Let $\chi \in C^{\infty}_0(-1,1)$ and $\phi\in C^{\infty}_0(B)$ be two cut-off functions such that $0\leq \chi,\phi\leq 1$, $\chi \equiv 1$ in $(-(1/2)^2,(1/2)^2)$ and $\chi \equiv 0$ in $(-1,1)\setminus (-(3/4)^2,(3/4)^2)$. And, similarly $\phi \equiv 1$ in $B(1/2)$ and $\phi \equiv 0$ in $B\setminus B(3/4)$. Now, set for simplicity $\psi(x,t) := \chi(t)\phi(x)$ and introduce the functions $v_i := u_i \psi$ and $F_i := -(u\cdot \nabla u_i + u_i \dvg u)\psi - (2 \nabla u_i\cdot \nabla \psi + u_i \Delta \psi) + u_i \partial_t \psi$ (with $i=1,2,3$). We have, at least in the sense of distributions in $Q_+$ that
\[  \partial_t v_i - \Delta v_i = F_i, \]
with $F_i \in L_2(Q_+)$. We define 
\[
\bar{F}_i(x,t) := \begin{cases}
F_i(x_1,x_2,x_3,t) \mbox{ in }\{ x_3 > 0 \}\\
-F_i(x_1,x_2,-x_3,t) \mbox{ in }\{ x_3 \leq 0 \}
\end{cases}
\]
the odd extension of $F_i$ and consider the initial boundary value problem
\[
\begin{cases}
\partial_t \bar{v}_i - \Delta \bar{v}_i = \bar{F}_i \mbox{ in } Q\\
\bar{v}_i|_{\partial'Q} = 0.
\end{cases}
\]
Standard parabolic theory insure the existence of a unique solution $\bar{v}_i$ that satisfies the following estimate
\[
\|\bar{v}_i\|_{W^{2,1}_2(Q)} \leq c \|\bar{F}_i\|_{L_2(Q)},
\]
where $c$ is an absolute positive constant. This uniqueness of $\bar{v}_i$ together with the parity of $\bar{F}_i$ ensure us that $\bar{v}_i$ is also odd. From this, we deduce without too much difficulty that 
\begin{equation}\label{VI-E44}
v_i = \bar{v}_i|_{\{x_3 > 0\}},
\end{equation}
which gives us $v_i \in W^{2,1}_2(Q_+)$ and by embedding $v_i \in W^{1,0}_{p_1}(Q_+)$ where \[p_1 = \frac{5\times 2}{5 - 2} >2; \]
thus we have that $F_i \in L_{p_1}(Q_+(1/2))$. Starting again the above machinery, but this time with $F_i \in L_{p_1}(Q_+(1/2))$ (the cut-off functions have to be changed into a suitable manner), and iterating, we end up with 
\[ F_i \in W^{1,0}_p(Q_+(a_p)), \]
with $0<a_p<1$ and for all $p\in [1,\infty)$. Going back to the equation of $\bar{v}_i$, we have at least in the sense of distribution that 
\[ 
\partial_t \bar{v}_{i,j} - \Delta \bar{v}_{i,j} = \bar{F}_{i,j}
\]
where $j=1,2,3$. But we have, if we fix $p>5/2$, that $\bar{F}_{i,j} \in L_p(Q(a))$, and by standard parabolic theory we obtain that $\bar{v}_{i,j} \in C^{\alpha,\alpha/2}(Q(a/2))$ with $\alpha = 2 -5/p >0$. Next, using \eqref{VI-E44}, we get that 
\[ 
v_{i,j} \in C^{\alpha,\alpha/2}(Q_+(a/2)),
\]
for $i,j=1,2,3$ and $0<a<1$. Next, let us notice that by choosing appropriately the cut-off functions, we have now that $F_i \in C^{\alpha,\alpha/2}(Q(a/4))$; and therefore, we get that (even if it means to change the domain on which we solve the equation of $\bar{v}_i$ into $Q(a/4)$):
\[ \bar{v}_i \in C^{\alpha+2,\alpha/2+1}(Q(a/4)), \]
thanks to Schauder's estimates. Using once more \eqref{VI-E44}, we obtain that 
\[ v_i \in C^{\alpha+2,\alpha/2+1}(Q_+(a/4))\quad \mbox{and similarly as before } F_{i,j} \in C^{\alpha,\alpha/2}(Q(a/8)) \]
Repeating this process, we have that the theorem is proved.
\section{Proof of Theorem~\ref{Thm3.4}}
Before continuing our development, let us point out that because of Theorem~\ref{Thm3.2} i.e. our version of the Caffarelli-Kohn-Nirenberg, we have that if $u$ admits singular points those points should necessary belong to the set $\{0\}\times (-1,0)$ (this is due to the fact that the $1$-dimensional parabolic Hausdorff measure of the set of singular points of $u$ in $Q$ is equal to zero; we have the same consequence for our model by following line by line the proof of that statement in \cite{Caff82}).\\
We recall that a point $z_0$ is a regular point of $u$ if there exists $\varrho >0$ such that $u$ is H\"older continuous in $Q(z_0,\varrho)$. And a point $z_0$ is a singular point if it is not a regular one.\\
Assume now that $z_{t_0}=(0,t_0)$ (with $-1<t_0<0$) is a singular point of $u$. Making use again of Theorem~\ref{Thm3.2}, we may construct, upon use of a space-time shift and using the natural scaling of \eqref{Toy-NS1} (i.e. $u^{\lambda}(x,t) := \lambda u(\lambda x,\lambda^2 t)$ with $\lambda>0$ is also a solution to \eqref{Toy-NS1} if $u$ is) a function $\tilde{u}$ with the following properties:
\begin{enumerate}
    \item $\tilde{u} \in L_{2,\infty}(Q)\cap W^{1,0}_2(Q)$ and obey \eqref{Toy-NS1} in $Q$ in the sense of distribution;
    \item $\tilde{u}\in L_{\infty}(B\times (-1,-a^2))$ for all $a\in (0,1)$;
    \item for all $r_1\in(0,1)$ such that $\tilde{u}\in L_{\infty}(\{r_1<|x|<1\}\times (-1,0))$;
    \item $\tilde{u}(x,t) = -\tilde{v}(|x|,t)x$.
\end{enumerate}
To see the previous assertion, we proceed in the following manner. Because of the observations made at the begin of this section, we have that there exists $r>0$ such that the first three points hold true in $Q(z_{t_0},r)$. Define now $\tilde{u}(x,t) = r u(r x,t_0 + r^2 t)$ (with $(x,t)\in Q$) and we steadily get that the new function satisfies all the above points. Moreover the origin $z=0$ is a singular point of $\tilde{u}$.\\
Next, we have the following proposition.
\begin{proposition}
The solution $\tilde{u}$ constructed above has the following property
    \[ 
    \sup_{z\in \overline{Q(1/2)}}|x|^{2/3}|u(x,t)| < \infty. 
    \]
\end{proposition}
A straightforward consequence of the previous proposition is that the origin is actually a regular point for $\tilde{u}$. Indeed, one can steadily show that $\tilde{u}\in L_{3,\infty}(Q(1/2))$ which is necessary condition for regularity. We will present the proof of this in a forthcoming paper; the proof essentially relies on an application of the backward uniqueness and unique continuation (introduced in \cite{Esc03}) to the system \eqref{Toy-NS1}. $z=0$ being a regular point of $\tilde{u}$ is a contradiction. Consequently, we have that $z=(0,t_0)$ is a regular point and this conclude the proof of Theorem~\ref{Thm3.4}. The only thing left is to prove the proposition.
\begin{proof}
 For simplicity, we drop in the sequel, the tilde symbol for $u$ and $v$. We steadily have the following equation for $v$
    \begin{equation}\label{RadTNVS}
        v_t = v_{rr} + \frac{4}{r}v_r + \frac{3}{2}r v v_r + \frac{5}{2}v^2
    \end{equation}
    for $(r,t)\in (0,1)\times (-1,0)$. Let us introduce the function
    \[ v^{(1)}(r,t) = r^{1+2/3}v(r,t) \]
    We have the following equation for $v^{(1)}$
    \begin{equation}\label{Pt2-E3}
        v^{(1)}_t - v^{(1)}_{rr} - (\frac{4}{3 r} + \frac{3}{2 r^{2/3}}v^{(1)})v^{(1)}_r + \frac{20}{9 r^2} v^{(1)} = 0,
    \end{equation}
    for $(r,t)\in (0,1)\times (-1,0)$. Let $a\in (0,1/2)$ and $\varepsilon \in (0,1/2)$. Our goal now is to apply a weak maximum principle to \eqref{Pt2-E3} in $(\varepsilon,1/2)\times (-1/4,-a^2)$; indeed, notice that $v$ is smooth in the closure of $(\varepsilon,1/2)\times (-1/4,-a^2)$. We have that 
    \begin{multline*}
        \max_{(r,t)\in [\varepsilon,1/2]\times [-1/4,-a^2]} |v^{(1)}(r,t)| = \max\left\{ \max_{\varepsilon\leq r\leq 1/2}|v^{(1)}(r,-\frac{1}{4})|,\max_{-1/4\leq t \leq -a^2}|v^{(1)}(\frac{1}{2},t)|, \right.\\
        \left. \max_{-1/4\leq t \leq -a^2}|v^{(1)}(\varepsilon,t)| \right\}
    \end{multline*}
    But from the second and third point of the properties we enumerated for $\tilde{u}$ above and by noticing that $v^{(1)}(0,t) = 0$ for all $t\in (-1/4,-a^2)$, we deduce that there exists a finite positive constant $C=C(u)$ independent of $a$ such that
    \[ \max_{(r,t)\in [0,1/2]\times [-1/4,-a^2]} |v^{(1)}(r,t)| \leq C, \]
    for all $a\in (0,1/2)$. Consequently, we have that  
    \begin{equation}
        \max_{(r,t)\in [0,1/2]\times [-1/4,0]} |v^{(1)}(r,t)| \leq C,
    \end{equation}
    which concludes the proof.
\end{proof}
\paragraph{Concluding remarks:} We conclude this paper by mentioning a consequence of Theorem~\ref{Thm3.1}. The latter gives us a refinement of the set of singular points of the solution $u$ (this is systematically true as soon as one have a higher integrability of the gradient of $u$). Indeed, we have, thanks to H\"older's inequality, that 
\begin{equation}\label{VII-E48}
    \frac{1}{r}\int_{Q(r)}|\nabla u|^2 dz \leq c(\delta) \left( \frac{1}{r^{1-2 \delta}}\int_{Q(r)}|\nabla u|^{2+\delta} dz \right)^{\frac{2}{2 + \delta}};
\end{equation}
consequently, Theorem~\ref{Thm3.2} guarantees that there exists a constant $\epsilon_1 = \epsilon_1(\delta)>0$ such that if 
\[ \sup_{0<r<1}\frac{1}{r^{1-2 \delta}}\int_{Q(r)}|\nabla u|^{2+\delta} dz < \epsilon_1, \]
then $z=0$ is a regular point. From this, we derive in a similar way as done in \cite{Caff82} that the $(1-2\delta)$-dimensional parabolic Hausdorff measure of the set of singular points of suitable weak solutions $u$ such that $u\in L_{\infty}(-1,0;BMO^{-1}(B))$ in $Q$ is null. If $\delta \geq 1/2$ then it is easy to see from \eqref{VII-E48} that $u$ is regular in $\overline{Q(1/2)}$. Unfortunately, we are not able to prove or disprove the same result for $\delta < 1/2 $ at the moment.
\paragraph{Acknowledgement}
This work was supported by the Engineering and Physical Sciences Research Council [EP/L015811/1]. The author would like to thank Gregory Seregin for the insightful discussions during the completion of this paper. 

\newpage

\bibliography{PaperI.bib}
\bibliographystyle{plain}
\end{document}